\documentclass[12pt]{amsart}
\usepackage{amssymb,amsmath,amsthm,comment}
\usepackage{pdfpages,graphicx} 
\usepackage[section]{placeins} 
\usepackage{wrapfig}
\usepackage{float}
\usepackage{lineno}

\usepackage{setspace}
\newcommand{\shrinkmargins}[1]{
  \addtolength{\textheight}{#1\topmargin}
  \addtolength{\textheight}{#1\topmargin}
  \addtolength{\textwidth}{#1\oddsidemargin}
  \addtolength{\textwidth}{#1\evensidemargin}
  \addtolength{\topmargin}{-#1\topmargin}
  \addtolength{\oddsidemargin}{-#1\oddsidemargin}
  \addtolength{\evensidemargin}{-#1\evensidemargin}
  }
\shrinkmargins{0.5}

\newtheorem{theorem}{Theorem}
\newtheorem{lemma}[theorem]{Lemma}

\newtheorem{corollary}[theorem]{Corollary}
\newtheorem*{theorem*}{Theorem}

{Claim}

\theoremstyle{definition}

\theoremstyle{remark}

\newtheorem*{remark}{Remark}
\newtheorem*{remarks}{{\bf Remarks}}
\newtheorem*{example}{Example}
\numberwithin{theorem}{section} \numberwithin{equation}{section}

\usepackage{caption}
\usepackage{multirow}

\makeatletter
\@namedef{subjclassname@2020}{\textup{2020} Mathematics Subject Classification}
\makeatother

\newcommand{\stirl}[2]{{#1 \brack #2}}

\newcommand{\norm}[2][]{\ensuremath{\left|\!\left|#2\right|\!\right|_{#1}}}

\newcommand{\floor}[1]{\left\lfloor #1 \right\rfloor}

\begin{document}
\title[Central Limit Theorems]
{Variations of Central Limit Theorems\\ and Stirling numbers of the First Kind}
\author{Bernhard Heim }
\address{Lehrstuhl A f\"{u}r Mathematik, RWTH Aachen University, 52056 Aachen, Germany}
\email{bernhard.heim@rwth-aachen.de}
\author{Markus Neuhauser}
\address{Kutaisi International University, 5/7, Youth Avenue,  Kutaisi, 4600 Georgia}
\address{Lehrstuhl A f\"{u}r Mathematik, RWTH Aachen University, 52056 Aachen, Germany}
\email{markus.neuhauser@kiu.edu.ge}
\subjclass[2020] {Primary
05A16, 60F05;
Secondary 05A16, 
11B39}
\keywords{Central Limit Theorem, Local Limit Theorem, Probabilistic Number Theory, Singularity Analysis.}
\begin{abstract}
We construct a new parametrization of double sequences $\{A_{n,k}(s)\}_{n,k}$ between
$A_{n,k}(0)= \binom{n-1}{k-1}$ and $A_{n,k}(1)= \frac{1}{n!}\stirl{n}{k}
$, 
where $\stirl{n}{k}$ are the unsigned Stirling numbers of the first kind.
For each $s$ we prove a central limit theorem and a local limit theorem.
This extends the de\,Moivre--Laplace central limit theorem and 
Goncharov's result, that unsigned Stirling numbers of the first kind are asymptotically normal.
Herewith, we provide several applications.
\end{abstract}
\maketitle
\section{Introduction}
The central limit theorem is one of the most remarkable theorems in science \cite{Fe45, Fe71, Fi11}. 
From the de\,Moivre--Laplace theorem in combinatorics and probability, involving binomial distributions,
the central limit theorem culminated in a
universal law of nature (\cite{KM18}, section $3$). In this paper,
we construct a parametrization of double sequences $\{A_{n,k}(s)\}_{n,k}$ between
$A_{n,k}(0)= \binom{n-1}{k-1}$ and $A_{n,k}(1)= \frac{1}{n!}\stirl{n}{k}
$, 
where $\stirl{n}{k}$ are the unsigned Stirling numbers of the first kind.
Finally, we apply Harper's method \cite{Ha67} and prove
a central limit theorem and a local limit theorem
for each $s$.

Let $Z_n \in \{0,1, \ldots, n\}$ denote a random variable with binomial distribution
\begin{equation*}
\mathbb{P}\left( Z_n = k\right) = \binom{n}{k} \, p^k \, q^{n-k},
\end{equation*}
where $0 < p <1$ and $q=1-p$. The central limit theorem
by de\,Moivre--Laplace states, that the normalized random
variables $Z_n$, converge in distribution 
against the standard normal distribution $
{N}\left( 0,1\right) $:
\begin{equation*}
\frac{Z_n - n\, p }{\sqrt{n \, p \, q}}
\stackrel{D}{\longrightarrow }
{N}(0,1).
\end{equation*}
In asymptotic analysis \cite{Be73, Ca15}, one is interested in the asymptotic normality of sequences.
Goncharov \cite{Go44,Go62} proved in 1944, that the unsigned Stirling numbers of the first kind
$\stirl{n}{k}$
have this property. 
More than $20$ years later, Harper \cite{Ha67} discovered a more conceptional method, proving also that the 
Stirling numbers of the second kind are asymptotically normal.
When
describing
Goncharov's proof,
Harper writes
``Goncharov \dots\ by brute force torturously manipulates the characteristic functions of the 
distributions until they approach
$\exp \left( -x^{2}/c\right) $, $c$ a positive constant."

Recently, Harper's method has been applied by Gawronski and Neuschel \cite{GN13}, investing in 
Euler--Frobenius numbers and by
Kahle and Stump \cite{KS19}. Note, that in several cases, also other 
non-gaussian distributions need to be considered (e.~g.\ limiting Betti distributions of Hilbert
schemes of $n$ points,
where the Gumbel distribution was the right limit distribution, 
Griffin et al.\ \cite{GORT22}).

Moreover, related to the topic, since we deal with unimodal sequences, we suggest the
analysis of properties of the modes.
We utilize a result of Darroch \cite{Da64, Be96} and study the modes of $\{A_{n,k}(s)\}$. This is
connected to Erd\H{o}s'
proof  \cite{Er53} of the Hammersley conjecture \cite{Ha51}, related to the peaks of $\{\stirl{n}{k}\}$.
For
other
sequences, we refer to Bringmann et al.\ \cite{BJSMR19}.

\section{Main results}
Let $h_s(n):= n^s$ for $s \in \mathbb{R}$. Let $P_n^{h_s}(x):= \frac{x}{h_s(n)} \, \sum_{k=0}^{n-1} P_k^{h_s}(x)$ with
initial value $P_0^{h_s}(x):=1$. We are interested in the coefficients of these polynomials:
\begin{equation*}
P_n^{h_s}(x) = \sum_{k=0}^n   A_{n,k}(s) \, x^k.
\end{equation*}
Let $n \geq 1$. Then $A_{n,n}\left( s
\right) = (n!)^{-s}$ and $A_{n,0}(s)=0$.
We deduce from \cite{HN21}, Example~1
and Example 3 for the values $s=0$ and $s=1$:
\begin{equation}\label{Raender}
A_{n,k}(0)= \binom{n-1}{k-1} \text{ and } A_{n,k}(1) = \frac{
1}{n!}\stirl{n}{k}.
\end{equation}
The unsigned Stirling numbers of the first kind $\stirl{
n}{
k}
$, denote the number of all permutations of a set of $n$ elements with
exactly
$k$ distinct cycles. We refer to B\'{o}na \cite{Bo12}.

We mainly focus on  $s \in [0,1]$ due to (\ref{Raender}).
Then, we
obtain central and local limit theorems for
the double sequence $\{A_{n,k}(s)\}_{n,k}$. 
The case $s \in [-1,0]$ can be reduced to the case $[0,1]$, since
\begin{equation*}
P_n^{h_s}(x) = h_{s}\left( n\right) \prod_{k=1}^{n
} h_s(k) \,\, x^{n+1} \, P_n^{h_{-s}}(x^{-1}).
\end{equation*}
\subsection{Central
limit
theorem}
The classical central limit theorem of
de\,Moivre (1738) and Laplace (1812), was
developed from the results in
probability theory \cite{Fi11, KM18} to a general theorem, without direct reference to
concepts as random variable, expected value,
and variance. We refer to Feller \cite{Fe45} and Canfield \cite{Ca15} for excellent
surveys. 
The modern version of the central limit theorem can also be considered as a theorem on the
asymptotic normality of a sequence of non-negative numbers in singularity analysis \cite{Be73}.
In this spirit, we state our first result in the most general form.
\begin{theorem}\label{main:CLT} 
Suppose $s \in [0,1]$.  
Then
there exist real sequences $\{a_n(s)\}_n$ and $\{b_n(s) \}_n$ with $b_n(s)$ positive for almost all $n$, so
that
\begin{equation}\label{CLT}
\lim_{n \to \infty}  \sup _{x \in \mathbb{R}} \left|
{ \frac{1}{P_n^{h_s}(1)}
\sum_{ k \, \leq \, a_n(s)  + x \, b_n(s)} A_{n,k}(s) - \frac{1}{\sqrt{2 \pi}} \int_{-\infty}^x
\mathrm{e}^{-\frac{t^2}{2}} \, \mathrm{d}t} \right|=0.
\end{equation}
\end{theorem}
Theorem \ref{main:CLT} is proven with Harper's method \cite{Ha67}. The sequences $\{a_n(s)\}_n$ and $\{b_n(s) \}_n$ are
provided by the expected values and variances of a suitable sequence of random variables. Let $h_s(0):=0$.
It is essential, that
\begin{equation}\label{realroots}
P_n^{h_s}(x) = (n!)^{-s} \, \prod_{k=0}^{n-1} \big( x + h_s(k) \big)
\end{equation}
has real roots $- h_s(k) \leq 0$.

Further, we utilize the Berry--Esseen theorem \cite{Ca15} to control the convergence rate.
The corresponding
expected values and variances are
\begin{eqnarray} 
\mu_n(s) & := &    1+  \sum_{k=1}^{n-1} \frac{1}{1+k^s},       \label{mu}\\
\sigma_n^2(s) & := &  \sum_{k=1}^{n-1} \frac{k^s}{\left(1+k^s\right)^2}. \label{sigma}
\end{eqnarray}
Then $\mu_n(0)= \frac{n+1}{2}$ and $\mu_n(1)=H_n$, where $H_n$ is the $n$th harmonic number. 
Moreover, let generally
$H_n^{(r)}:= \sum_{k=1}^{n} k^{-r}$.
We have $\sigma_n^2(0)= \frac{n-1}{4}$ and $\sigma_n^2(1) = H_n - H_n^{(2)}$.
Most importantly, let $s \in [0,1]$.
Then 
\begin{equation*}
\lim_{n \rightarrow \infty} \sigma_n(s) \longrightarrow \infty.
\end{equation*}
Therefore, we obtain:
\begin{theorem}\label{explicit} 
Suppose $s \in [0,1]$. There exists a positive constant $C$ so
that
\begin{equation*}
\lim_{n \to \infty}  \sup _{x \in \mathbb{R}} \left|
{ \frac{1}{P_n^{h_s}(1)}
\sum_{ k \, \leq \, \mu_n(s)  + x \, \sigma_n(s)} A_{n,k}(s) - \frac{1}{\sqrt{2 \pi}} \int_{-\infty}^x
\mathrm{e}^{-\frac{t^2}{2}} \, \mathrm{d}t} \right| \leq C \, \frac{1}{\sigma_n(s)}.
\end{equation*}
The standard deviation $\sigma_n(s)$
approaches infinity.
\end{theorem}
\begin{remarks} \ \\a) Theorem \ref{explicit} implies Theorem \ref{main:CLT}. In singularity analysis \cite{Be73, Ca15}, one
is most interested in the asymptotic behaviors
of $\{a_n\}_n$ and $\{b_n\}_n$, rather
than in the concrete realization,
as given in Theorem \ref{explicit}. \\
b) The constant be chosen as $C=0.7975$ (we refer to \cite{VB72}, and the survey article \cite{Pi97}). \\
c) Let $s >1$. Then $\lim_{n
\rightarrow \infty} \sigma _{n}\left( s\right)
\leq
\lim _{n\rightarrow \infty }\sum _{k=1}^{n-1}k^{-s}=\zeta \left( s\right) <\infty $. Here we denote by $\zeta(s)$ the Riemann zeta function.
\end{remarks}

\subsection{Local
limit
theorem}
We refer to Section~\ref{four} for an introduction. We prove:
\begin{theorem} 
\label{local}
Let $s \in \mathbb{R}$. Then there exists a universal constant $K>0$, so
that 
\begin{equation*}
\max _k \left| \frac{\sigma_n(s)}{P_n^{h_s}(1)}  \, A_{n,k}(s) - 
\frac{\mathrm{e}^{-\frac{(x_n(s))^2}{2}}}
{\sqrt{2 \pi}}\right| < \frac{K}{\sigma _{n}\left( s\right) },
\end{equation*}
for $x_n(s) = (k - \mu_n(s))/ \sigma_n(s)$.
Uniformly for $k-\mu_n(s) = O(\sigma_n(s))$ we have
\begin{equation*}
\frac{ A_{n,k}(s)}{P_n^{h_s}(1)} \sim   \frac{  \mathrm{e}^{-\left( x_{n}\left( s\right) \right) ^{2}/2}}{ \sigma_n(s) \, \sqrt{2 \pi}}.
\end{equation*}
\end{theorem}
\subsection{Peaks and
plateaux}
The polynomials $P_n^{h_s}\left( x\right) $ are real-rooted. Therefore, a theorem by Newton implies, that
the sequence $\{A_{n,k}(s)\}_k$ is unimodal and has
two modes at most. Either we have one peak, or a plateau.

In the case $s=0$,
we have a peak for $n$ odd at $k=\frac{n+1}{2}$ and a plateau for $n$ even at
$k= \frac{n}{2}$ and $\frac{n+2}{2}$. This is obvious, since the $A_{n,k}(0)$ are binomial coefficients.
The case $s=1$ is more delicate. Let $n \geq 3$. Hammersley \cite{Ha51} conjectured in the context of
Stirling numbers of the first kind that there is always a peak.
This
was
proved by Erd\H{o}s \cite{Er53}. The proof depends on the fact, that
$\left\{ \stirl{
n}{
k}
\right\} _k$ are natural numbers. This allows Erd\H{o}s to apply special results
related to the prime number theorem and certain divisibility properties of the Stirling numbers of the first kind.
Our goal is to contribute to
the case $s\in (0,1)$ and obtain information for $s=1$. But this seems to be
very difficult, since in general,
the numbers $A_{n,k}(s)$ are
not integers.
Nevertheless, by utilizing a
theorem
by Darroch \cite{Da64} we obtain:
\begin{theorem} \label{peak}
Let $n \geq 6$. Let $k_0 \in \mathbb{N}$ be any integer
associated with $s \in (0,1)$ so
that
\begin{equation*}
k_0 = 1 + \sum_{k=1}^{n-1} \frac{1}{1 + k^s}.
\end{equation*}
Then the sequence
$\left\{ A_{n,k}\left( s\right) \right\} $ has a peak at $k_0$.
The number of possible $k_0$ is given by the number of integers between
$H_n$ and $\frac{n+1}{2}$.
\end{theorem}
\section{The 
probabilistic
viewpoint on
asymptotic
normality}
We begin with a useful tool from probability.
\subsection{The Berry--Esseen
theorem \cite{Ca15},
theorem 3.2.4}
Let $X$ be a random variable. We denote by $\mathbb{E}(X)$ and $\mathbb{V}(X)$ the expected value
and variance of $X$.
\begin{theorem}\label{th:BE}
Let $X_{n,k}$ for $1 \leq k \leq n$ be independent random variables with values in $\{0,1,\ldots, n\}$. Let
$\mu_{n,k}$ be the expected values, $\sigma_{n,k}^2$ the variances,
and
$$\rho_{n,k} = \mathbb{E} (\vert X_{n,k} - \mu_{n,k} \vert^3) < \infty$$ 
the absolute third central moment. 
Let
$\mu_n := \sum_{k=1}^n \mu_{n,k}$, $\sigma_n^2:= \sum_{k=1}^n \sigma_{n,k}^2$,
and $Z_n:= \sum_{k=1}^n X_{n,k}$.
Let $Z_n^{*}= (Z_n - \mu_n)/\sigma_n$. Then
\begin{equation*}
{\norm{\mathbb{P}\left( Z_{n}^{\ast } < x\right) - \frac{1}{\sqrt{2 \pi}} 
\int_{-\infty}^x 
\mathrm{e}^{-\frac{t^2}{2}} \, \mathrm{d}t}}_{\mathbb{R}}      \leq C \,\,
\frac{\sum_{k=1}^n \rho_{n,k}}{\sigma_n^3},
\end{equation*}
where ${\norm{f(x)}}_{\mathbb{R}}$ denotes the supremum norm of $f$ on $\mathbb{R}$
and $C >0$ is a universal constant. This constant can be chosen as $C=0.7975$ \cite{VB72}.
\end{theorem}
Let $P_n(x)= \sum_{k=0}^n a_{n,k} \, x^k$ 
be a monic polynomial of degree $n$ with $ a_{n,k} \geq 0$.
Suppose, the roots of $P_n(x)$ are real and $P_n(x) = \prod_{k=1}^n \left( x + r_k \right)$.
Harper \cite{Ha67} introduced a triangular array of Bernoulli random variables $X_{n,j}$ with
distribution 
\begin{equation*}
\mathbb{P}\left( X_{n,j}=0\right) := \frac{r_j}{1 + r_j} \text{ and }
\mathbb{P}\left( X_{n,j}=1\right) := \frac{1}{1 + r_j}.
\end{equation*}
Let $Z_n:= \sum_{j=1}^n X_{n,j}$. Then $\mathbb{P}\left( Z_{n}=k\right) = \frac{a_{n,k}}{P_n(1)}$.
Let $X_{n,j}$ be given. Then
\begin{equation*} 
\mathbb{E}(X_{n,j}) = \frac{1}{1+r_j}, \,\, \, 
\mathbb{V}(X_{n,j}) = \frac{r_j}{( 1 + r_j)^2},\, \,\,
\mathbb{E} ( \vert X_{n,j} - \mathbb{E}(X_{n,j})\vert^3) =  \frac{r_j(1+ r_j^2)}{(1 + r_j)^4}.
\end{equation*}
This implies that
\begin{equation} 
\label{inequality}
\mathbb{E} ( \vert X_{n,j} - \mathbb{E}(X_{n,j})\vert^3) < \mathbb{V}(X_{n,j}).
\end{equation}
We apply Harper's method setting
$P_n(x)= (n!)^s \,\, P_n^{h_s}(x)$.
The
expected value of $Z_n(s)$ is given by $\mu_n(s)$ and the variance by $\sigma_n(s)^2$,
as recorded in (\ref{mu}) and (\ref{sigma}).

\begin{lemma}
a) Let $s \in [0,1)$.
Then
\begin{equation*}
\mu_n(s) \sim \frac{n^{1-s}}{1-s} \text{ and } \sigma_n(s) \asymp n^{1-s}.
\end{equation*}
b)
We obtain
\begin{equation*}
\mu_n(1) \sim \ln \, n \text{ and } \sigma_n^2 (1) \asymp \ln \, n.
\end{equation*}
\end{lemma}
Recall that $f \asymp g$ is an
abbreviation for $f=
{O}(g)$ and $g =
{O}(f)$.

\begin{proof}
a) Let $n\geq 2$. We have
$
\frac{1}{1+k^{s}}\leq \frac{1}{k^{s}}$.
For $k
\leq t\leq
k+1
$ we have
$\left(
k+1\right)
^{-s}\leq t^{-s}\leq
k
^{-s}$. Therefore,
$
1+\sum _{k=1}^{n-1}\frac{1}{1+k^{
s}}\leq 2+\int _{1}^{n-1}t^{-s}\,\mathrm{d}t$.
Now, for $0\leq s<1$, we obtain
$
\sum _{k=1}^{n-1}\frac{1}{1+k^{
s}}\leq
2+\frac{\left( n-1\right) ^{1-s}-1}{1-s}$.

For $k\geq K$, we have $\frac{1}{1+k^{s}}\geq \frac{1}{K^{-s}+1}k^{-s}$.
Therefore,
$1+\sum _{k=1}^{n-1}\frac{1}{1+k^{s}}\geq \frac{1}{K^{-s}+1}\int _{K}^{n}t^{-s}\,\mathrm{d}t$.
For $0\leq s<1$, this shows
$1+\sum _{k=1}^{n-1}\frac{1}{1+k^{s}}\geq \frac{1}{K^{-s}+1}\frac{n^{1-s}-K^{1-s}}{1-s}$.

This
indicates that for $0\leq s<1$ and
arbitrarily small $\varepsilon >0$, there is an
$N\in \mathbb{N}$ so
that
$\left( 1-\varepsilon \right) \frac{n^{1-s}}{1-s}\leq \mu _{n}\left( s\right)
\leq \left( 1+
\varepsilon \right) \frac{n^{1-s}}{1-s}$ for all
$n \geq N$.
Therefore, for $0\leq s<1$ and
all $\varepsilon >0$, there is an $N\in \mathbb{N}$,
such that for all $n\geq N$ holds
$1-\varepsilon \leq \frac{\mu _{n}\left( s\right) }{\frac{n^{1-s}}{1-s}}\leq 1+\varepsilon$.

We have
$\sigma _{n}^{2}\left( s\right) \geq \sum _{k=1}^{n-1}\frac{1}{4k^{s}}\geq \frac{1}{4}\int _{1}^{n}t^{-s}\,\mathrm{d}t$.
So, for $0\leq s<1$, we obtain
$\sigma _{n}^{2}\left( s\right) \geq \frac{n^{1-s}-1}{4-4s}$.
Additionally, we have
$\sigma _{n}^{2}\left( s\right) \leq \sum _{k=1}^{n-1}k^{-s}=\mu _{n}\left( s\right) $,
meaning that the same upper bounds also apply here.

b) Similarly, for $s=1$, we obtain
$\frac{\ln \left( n\right) -\ln \left( K\right) }{K^{-s}+1}\leq
1+\sum _{k=1}^{n-1}\frac{1}{1+k^{
s}}\leq 2+\ln \left( n-1\right) $,
$1-\varepsilon \leq \frac{\mu _{n}\left( 1\right) }{\ln \left( n\right) }\leq 1+\varepsilon $,
and
$\sigma _{n}^{2}\left( 1\right) \geq \frac{1}{4}\ln \left( n\right) $.
\end{proof}
\begin{corollary} \label{infty}
Let $s \in [0,1]$.
Then $\lim_{n \to \infty} \sigma_n(s) = \infty$.
\end{corollary}
\begin{remark}
Let $0 \leq s_1 \leq s_2 $. Then $\sigma_n(s_1) \geq \sigma_n(s_2)$, since $\frac{\mathrm{d}}{\mathrm{d}s} \sigma_n(s) <0$.
\end{remark}

\subsection{Proof of Theorem \ref{explicit} }

Let $s \in [0,1]$.
Let $X_{n,k}(s) \in \{0,1\}$ for $1 \leq k \leq n$ be a random variable with values $0$ and $1$. Here 
\begin{equation*}
\mathbb{P}\left (X_{n,k}(s) =1\right) = \frac{1}{ 1 + r_{n,k}(s)},
\end{equation*}
where $r_{n,k}(s) = h_s(k-1)$.
We put $Z_n(s):= \sum_{k=1}^n X_{n,k}(s)$. Then
\begin{equation*}
\mathbb{P}\left( Z_{n}\left( s\right) =k\right) =  \frac{A_{n,k}(s)}{P_n^{h_s}(1)}.
\end{equation*}
Thus, we obtain
\begin{equation*}
\mathbb{E}(Z_n(s)) = \mu_n(s) = 1 + \sum_{k=1}^{n-1}  \frac{1}{1 + k^s} 
\text{ and } \mathbb{V}(Z_n(s)) = \sigma_n(s)^2= \sum_{k=1}^{n-1} \frac{k^s}{( 1 + k^s)^2}.
\end{equation*}
Corollary \ref{infty} states, that the variance
approaches infinity.
This proves Theorem \ref{explicit}.
\subsection{Proof of Theorem  \ref{main:CLT}}
This follows from Theorem \ref{explicit}.
The crucial part of successfully applying
Harper's method and the Berry--Esseen
theorem is that
the variance
approaches infinity.

\section{\label{four}Local
limit
theorem}
A double indexed sequence $\{a_{n,k}\}_{n,k}$ satisfies a local limit theorem on a set $S$ of real numbers provided
\begin{equation*}
\sup _{x \in S} \left| \frac{ \sigma_n \, a(n, \floor{\mu_n + x \, \sigma_n})}{\sum_k a(n,k)} - \frac{\mathrm{e}^{- x^2/2}}{\sqrt{2 \pi}} \right| \longrightarrow 0
\end{equation*}
(cf.\ Canfield \cite{Ca15}, section 3.7). We recall the following result due to Bender \cite{Be73}.
\begin{theorem}[Bender]
Suppose, that the
$\left\{ a\left( n,k\right) \right\} _{k}$ for
$n\in \mathbb{N}$ are asymptotically
normal, and $\sigma_n^2 \rightarrow \infty$. If for each $n$ the sequence
$\left\{ a\left( n,k\right) \right\} _{k}$ is
log-concave in $k$, then
$\left\{ a\left( n,k\right) \right\} _{k}$
satisfies
a local limit theorem on $S= \mathbb{R}$.
\end{theorem}

\subsection{P\'olya
frequency
sequences and
limit
theorems}
We follow the excellent survey by Pitman \cite{Pi97} and apply
several results to the sequences $\{ A_{n,k}(s)\}_k$.

Let $(a_0,a_1, \ldots, a_n)$ be a
sequence of non-negative real numbers.
Let $P_n(x):= \sum_{k=0}^n \, a_k \, x^k$ be real rooted and $P_n(1)>0$.
Then the sequence is called a (finite) P\'olya frequency sequence.
Let $Z_n \in \{0,1, \ldots, n\}$ be a random variable with
$\mathbb{P}\left( Z_{n}=k\right) := \frac{a_k}{P_n(1)}$,
mean $\mu$, and variance $\sigma^2$.
Then 
\begin{equation*}\label{clt}
\max _k \left|             \mathbb{P}\left( 0 \leq Z_n \leq k\right) -
\frac{1}{\sqrt{2 \, \pi}} \, \int_{-\infty}^{\frac{k - \mu}{\sigma}}
\mathrm{e}^{- \frac{t^2}{2}} \, \mathrm{d}t
\right| < \frac{0.7975}{\sigma}
\end{equation*}
(see \cite{Pi97}, formula (24) or \cite{VB72} for a reference). 

Further, there exists a universal constant $K$ so
that
\begin{equation*}\label{llt}
\max _k \left|   \sigma \mathbb{P}\left( Z_n = k\right) -
\frac{1}{\sqrt{2 \, \pi}} \,
\mathrm{e}^{- \left(\frac{k - \mu}{\sigma}\right)^2 /2}
\right| < \frac{K}{\sigma}.
\end{equation*}
The bound is due to Platanov \cite{Pl80} (see also \cite{Pi97}, formula (25)).
\begin{proof}[Proof  of Theorem \ref{local}]
It follows from our previous considerations, that $\{A_{n,k}(s)\}_k$ is a P\'olya frequency sequence for all $s \in \mathbb{R}$.
This implies the
theorem.
\end{proof}

\section{Peaks of $\{A_{n,k}(s)\}_k$}
We recall
Darroch's theorem
\cite{Da64}. Let $(a_0,a_1, \ldots, a_n)$ 
be a P\'olya frequency sequence.
Let $p(x):= \sum _{k=0}^{n}a_{k}x^{k}$ with $p(1)>0$. Let $\mu_n:= \frac{p'(1)}{p(1)}$.
Newton's theorem implies, that the sequence $\{a_k\}_k$ is unimodal and has
two modes at most. 
Darroch proved, that the modes have
distance
less than $1$ from $\mu_n$.
Armed with the results of the previous sections we have:
\begin{proof} [Proof of Theorem \ref{peak}] \ \\
We consider the polynomials $P_n^{h_s}(x)$. Then $\mu_n(s) = 1 + \sum_{k=1}^{n-1} \frac{1}{1+k^s}$.
Suppose, that $\mu_n(s)$ is an integer.
Then $\{A_{n,k}(s)\}_k$ has a peak. Let us restrict $\mu_n$ to
$[0,1]$. In this case
\begin{equation*}
\mu_n: [0,1] \longrightarrow [ \mu_n(1), \mu_n(0)], \quad s \mapsto \mu_n(s).
\end{equation*}
It construes that $
\mu_n$ is differentiable and strictly monotone
for $n\geq 3$.
We have 
\begin{equation*}
\frac{\mathrm{d}}{\mathrm{d}s}
\mu _{n}\left( s
\right) = - \sum_{k=2}^{n-1} \frac{ k^s\ln k}{(1+k^s)^2} <0.
\end{equation*}
Therefore, $\mu_n$ is bijective. Let $n \geq 6$. Then $\mu_n(0) - \mu_n(1) >1$. This implies that integers $k_0 \in
\big(\mu_n(1), \mu_n(0)\big)$ exist and are realizable
by  suitable $s \in (0,1)$:
\begin{equation*}
k_0=1 + \sum_{k=1}^{n-1} \frac{1}{1+k^s}.
\end{equation*}
Let such an $s$ be given. Then $\{A_{n,k}(s)\}_k$ has a peak at $k_0$.
\end{proof}
We add the following approximation of $\mu _{n}\left( s\right) $.
\begin{remark}
Let $s\in \left[ 0,1\right] $ and $n \geq 3$.
There exists a $\xi \in [0,1]$ so
that
\begin{equation*}
\mu _{n}\left( s\right) = \frac{n+1}{2} - \left(\frac{1}{4} \sum_{k=2}^{n-1} \ln k\right) \,\, s+ \frac{\mu _{n}^{\prime \prime }\left( \xi \right) }{2} \,\, s^2.
\end{equation*}
\end{remark}

Finally, we provide an illustration of Theorem \ref{peak} by
Table~\ref{erdos}.
Let $m_n(1)$ be the unique mode of $\{A_{n,k}(1)\}$, as proven by Erd\H{o}s. 
\begin{table}[H]
$$
\begin{array}{|c|c|c|c|c|c|c|c|c|c|c|c|c|} \hline
n & 1&2&3&4&5 &6&7&8&9&10 &100&1000\\ \hline
H_n & 1 & 1.5 & 1.83 &2.08 & 2.28 & 2.45 & 2.59& 2.72 &  2.83& 2.93& 5.19& 7.49\\ \hline
m_n(1) & - & - & 2& 2  & 2& 2&2 & 3& 3 & 3& 5& 7
\\ \hline
\frac{n+1}{2} &1 & 1.5&2 & 2.5& 3& 3.5 & 4 &4.5& 5 &5.5 & 50.5 & 500.5 \\ \hline
\end{array}
$$
\caption{\label{erdos}Modes for $s=1$ and related values.}
\end{table}

\section{Applications}
\subsection{Approximation of Stirling numbers of the first kind}
Wilf \cite{Wi93} contributed
to the asymptotic
behavior of Stirling numbers of the first kind.
Several asymptotic formulas are provided. Wilf also compares his results with
Jordan's formula
in the case $\stirl{
100
}{5}
$ and presents numerical data. Although, the focus of this paper is not to obtain 
best approximations, the numerical value we obtain is already
solid. We refer to Theorem~\ref{local} and the approximation
\begin{equation}\label{asymp}
\frac{
1
}{100!}\stirl{100}{5}
\approx \frac{\mathrm{e}^{- \left(\frac{5 - \mu_{100}(1)}{\sigma_{100}(1)}\right)^2  /2}}{ \sigma_{100}(1) \sqrt{2 \pi}}.
\end{equation}
We have $\mu_{100}(1) \approx
5.19$ and $\sigma_{100}(1)
\approx 1.88477$.
Wilf considers the value of $\stirl{
100
}{5}
/99!$
(Table~\ref{wilf}).
\begin{table}[H]
\begin{tabular}{l|l|l|}
& $\stirl{
100
}{5
}/99! $ & Error \\ \hline
Exact value & $21.1204415 \ldots$ & - \\ \hline
Jordan formula & 18.740 \ldots & $
\approx 11 \%$ \\ \hline
$3$ term of Eq. ($9$) (Wilf) & $21.24986 \ldots$ & $\approx 0.613
\%$ \\ \hline
Theorem $1$ above (Wilf)& $20.960 \ldots $ & $\approx 0.76 \%$ \\ \hline
Approximation (\ref{asymp}) (this paper) & $21.062180 \ldots $ & $\approx 0.28 \%$ \\ \hline
Eq. ($7$) to order $1/n$ (Wilf) & $21.12070 \ldots $ & $\approx 0.0012 \%$ \\ \hline
Eq. ($7$) to order $1/n^2$ (Wilf) & $21.1204409  \ldots $ & $\approx 0.000003 \%$ \\ \hline
\end{tabular}
\caption{\label{wilf}Several approximations of the maximal value for $n=100$ (see Wilf \cite{Wi93}, page 349 for details).}
\end{table}

\subsection{One mode property of the Stirling numbers of the
first
kind}
Let $n \geq 3$. Erd\H{o}s \cite{Er53} proved that $\left\{ \stirl{
n}{
k}\right\} _k$ has one mode. We give a new proof for infinitely many $n$.
A variant of Darroch's
theorem \cite{Da64, Pi97} states: Let $\{a_k\}_k$ be a P\'olya frequency
sequence.
Let $\mu_n = \frac{p_n'(1)}{p_n(1)}$, 
then the sequence has exactly one mode $k_0$, if
\begin{equation*}
k_0 \leq \mu_n  < k_0 + \frac{1}{k_0+2} \, \,  \text{ or } \,\,k_0 - \frac{1}{n - k_0 +2} < \mu_n \leq k_0.
\end{equation*}
Let the sequence $A_{n,k}(1)$ be given. Then $\mu_n(1)= H_n$. Since $H_n$ is unbounded and $H_{n+1} = H_n + \frac{1}{n+1}$,
we can find infinitely many pairs $(n,k_0)$,
such that $H_n \in [k_0, k_0 + \frac{1}{k_0 +2}]$. This implies, that
$\left \{ A_{n,k}\left( 1\right) \right\} _{k}$
has one mode and thus, $\left\{ \stirl{
n}{
k}
\right\} _{k}$ has one mode.

\begin{example}
From Table~\ref{erdos} we see that for
$n=10$, we have $k_{0}=3$ and
$3-\frac{1}{9}<2.9<H_{10}=\mu _{10}\left( 1\right) \leq k_{0}$.
For $n=30$, we have $\mu _{30}=H_{30}\approx 3.995$ and therefore,
$k_{0}=4$ is the unique mode, as
$4-\frac{1}{28}<3.97<\mu _{30}\leq 4$. For $83\leq n\leq 95$, we have
$5.002\approx \mu _{83}=H_{83}\leq \mu _{n}\leq \mu _{95}\approx 5.136$
and therefore, $k_{0}=5$ is the unique
mode as $5\leq \mu _{n}<5.14<5+\frac{1}{7}$.
\end{example}

One can vary the argument and show:
\begin{theorem} Let $n \geq 6$. We consider the sequence $\{A_{n,k}(s)\}_k$ for $s \in (0,1)$.
Let $k_0= \mu_n(s_0)\in \mathbb{N}$ with $s_0 \in (0,1)$. Then there exists
an $\varepsilon >0$, so
that for all
$s \in (s_0 - \varepsilon, s_0 + \varepsilon)$,
the sequences $\{A_{n,k}(s)\}_k$ have exactly one mode.
\end{theorem}

\end{document}